\documentclass{article}

\usepackage{amsfonts}
\usepackage{amssymb}
\usepackage{amsmath}
\usepackage{amsthm}
\usepackage{ifthen}

\theoremstyle{plain}
\newtheorem{theorem}{Theorem}[section]
\newtheorem{corollary}[theorem]{Corollary}
\newtheorem{proposition}[theorem]{Proposition}
\newtheorem{conjecture}[theorem]{Conjecture}
\theoremstyle{definition}
\newtheorem{definition}[theorem]{Definition}
\theoremstyle{remark}

\theoremstyle{plain}
\numberwithin{equation}{section}

\newcommand{\parder}[3][Default]{
	\frac{\partial \ifthenelse{\equal{#1}{Default}}{}{^{#1}}#2}{
              \partial #3 \ifthenelse{\equal{#1}{Default}}{}{^{#1}}}}

\newcommand{\jac}{{\mathcal{J}}}
\newcommand{\GL}{\operatorname{GL}}
\newcommand{\tr}{\operatorname{tr}}
\newcommand{\Mat}{\operatorname{Mat}}
\newcommand{\rk}{\operatorname{rk}}
\newcommand{\chr}{\operatorname{char}}
\newcommand{\imp}{{\mathversion{bold}$\Rightarrow$ }}
\newcommand{\pmi}{{\mathversion{bold}$\Leftarrow$ }}
\newcommand{\C}{{\mathbb C}}
\newcommand{\F}{{\mathbb F}}
\newcommand{\N}{{\mathbb N}}
\newcommand{\p}{{\mathfrak p}}

\title{The strong nilpotency index of a matrix}
\author{Michiel de Bondt \\
	Radboud University, Nijmegen, The Netherlands \\
	Email: debondt@math.ru.nl}

\begin{document}

\maketitle

\begin{abstract}
\noindent
It is known that strongly nilpotent matrices over a division ring 
are linearly triangularizable. We describe the structure of such matrices 
in terms of the strong nilpotency index. We apply our results
on quasi-translation $x + H$ such that $\jac H$ has strong nilpotency index two.
\end{abstract}



\section{Introduction}

The Jacobian conjecture asserts that a polynomial
map $F$ over a field of characteristic zero has a polynomial inverse $G$ 
in case $\det \jac F$ is a unit of the base field. In \cite{bcw} and \cite{yagzhev},
it is show that the Jacobian Conjecture holds if it holds for all polynomial maps $x + H$
of such that $\jac H$ is nilpotent. This result is refined in \cite[\S 6.2]{Arnobook},
where it is shown that for any polynomial map $F$ of degree $d$ in dimension $n$,
the Jacobian conjecture holds if it holds for all polynomial maps $x + H$
of degree $d$ in dimension $(d-1)n$ such that $\jac H$ is nilpotent, which
subsequently holds if the Jacobian conjecture holds for all polynomial maps 
$x + H$ of degree $d$ in dimension $(d-1)n + 1$ such that $H$ is homogeneous of 
degree $d$.

So it is natural to look at nilpotent Jacobian matrices. For quadratic maps, 
it is already known that the Jacobian conjecture holds, but since we have $(d-1)n = n$
in that case, we see that reduction to the case that $F = x + H$ with $\jac H$ nilpotent
does not cost extra dimensions. But one may even assume that additionally, $H$ is 
homogeneous. That is why more than twenty years ago, Gary Meisters and
Czes{\l}aw Olech looked at nilpotent Jacobian matrices of quadratic homogeneous $H$ in low 
dimensions in \cite{mo}, and it appeared that in dimensions $n \le 4$, the Jacobian matrix 
of $H$ was so-called {\em strongly nilpotent}, but not necessarily in higher dimensions.

Somewhat later, but still more than fifteen years ago, 
polynomial maps $x + H$ over a field of characteristic zero such that $\jac H$
is strongly nilpotent were classified as so-called {\em linearly triangularizable}
polynomial maps, in both \cite{eh} and \cite{Yu}. Since linear triangularizable maps
are invertible, the Jacobian conjecture holds for polynomial maps $x + H$ over a field 
of characteristic zero such that $\jac H$ is strongly nilpotent.
But there are no results about the index of
strong nilpotency: just as with the nilpotency index for nilpotent matrices,
on can define the strong nilpotency index for strongly nilpotent matrices as the 
minimum number of factors such that its product is zero.

Actually, there are two equivalent definitions for strongly nilpotent matrices
over a field of characteristic zero, but both definitions no longer correspond 
when the base ring becomes arbitrary. Sometimes, the definitions only deal with
matrices which are Jacobians, but such a restriction is unnecessary, see also 
equation (\ref{chainrule}) in the beginning of the proof of corollary \ref{strongnilcor} 
below. It is neither necessary to assume that the base 
ring is commutative or even a field. Although the definitions do not require that the base 
ring is unital, all results in this article are about rings $R$ with unity.

The first of both definition is given in 
\cite{mo}. Write $f|_{x=g}$ for substituting $x$ by $g$ in $f$.

\begin{definition}[Meisters and Olech in \cite{mo}] \label{MOdef}
Let $R[x]$ be the polynomial ring over a ring $R$ (not necessarily commutative)
in $n$ indeterminates $x_1, x_2, \ldots, x_n$. A square matrix $M$ with entries in $R[x]$
is called \emph{strongly nilpotent} if for some $r \in \N$, we have 
\begin{equation} \label{MOeq}
M|_{x = v_1} \cdot M|_{x = v_2} \cdot \cdots \cdot M|_{x = v_r} = 0
\end{equation}
for all $v_i \in R^n$.
\end{definition}

\noindent
Actually, $r$ is equal to the dimension of the matrix $M$ in the original definition,
but I consider that a consequence of restricting to fields of characteristic zero, or
just to reduced commutative rings, see the later proposition \ref{strongnilyured}.

\begin{definition}[Van den Essen and Hubbers in \cite{eh}] \label{EHdef}
Let $R[x]$ be the polynomial ring over a ring $R$ (not necessarily commutative)
in $n$ indeterminates $x_1, x_2, \ldots, x_n$, which commute with each other and 
with elements of $R$. A square matrix $M$ with entries in $R[x]$
is called \emph{strongly nilpotent} if for some $r \in \N$, we have 
\begin{equation} \label{EHeq}
M|_{x = y^{(1)}} \cdot M|_{x = y^{(2)}} \cdot \cdots \cdot M|_{x = y^{(r)}} = 0
\end{equation} 
where the $y^{(i)}$ are other tuples of $n$ indeterminates, 
which commute with each other and with elements of $R$.
\end{definition}

\noindent
Again, $r$ is equal to the dimension of the matrix $M$ in the original definition, 
which is given in \cite{eh}. Over infinite integral domains, both definitions are the same. 
We call the minimum possible $r \in \N$ such that (\ref{MOeq}) or (\ref{EHeq}) 
respectively holds the \emph{strong nilpotency index} of $M$.

In section 2, we consider definition \ref{EHdef} because it does not correspond to definition 
\ref{MOdef} when $R$ is a finite field. This is because different polynomials can be the same as a 
function, whence substituting all possible elements of $R$ in the indeterminates will not 
distinguish both polynomials.
For instance, the polynomials $x_1$ and $x_1^q$ over $\F_q$ are the same as functions, so they act 
equal on substituting an element of $\F_q$ for $x_1$. Over infinite 
integral domains, different polynomials are always different function as well. So definitions
\ref{EHdef} and \ref{MOdef} are equivalent when $R$ is an infinite integral domain.

In \cite{Yu}, definition \ref{MOdef} has been generalized to functions to $R$ in general, not just
multivariate polyniomials, which we will discuss in section 3. Just like the main results in \cite{eh} 
and \cite{Yu} are similar, while the respective definitions \ref{EHdef} and the generalization of 
\ref{MOdef} are different, the refinements of these results by way of the strong nilpotency index, 
which are theorem \ref{strongnil} and theorem \ref{strongnilyu} respectively, are similar as well.

In the last section, we apply the results of section 2 on quasi-translations
of strong nilpotency index two. For more information about quasi-translations, we refer to the 
last section and its references.

\section{Strongly nilpotent matrices in the sense of \cite{eh}}

In \cite{eh}, the authors prove that strongly nilpotent Jacobians in the sense of 
\ref{EHdef} are linearly triangularizable, but there is no need to restrict to
Jacobians. Instead, one can look at any matrix of polynomials over a division ring
$D$. Notice that for finite fields, different polynomials might represent the same
function. Therefore, the generalization to polynomial matrices is not a special case of 
the generalization of definition \ref{MOdef} in \cite{Yu} that will follow in 
the next section.

\begin{theorem} \label{strongnil}
Let $n \ge 0$ and $x = x_1,x_2,\ldots,x_n$ be variables which commute with 
each other and with elements of a division ring $D$. Assume that $M \in \Mat_m(D[x])$,
where $m \ge 1$.

Then $M$ has strong nilpotency index $r$ (in the sense of definition \cite{eh}), if and only if 
there exists a $T \in \GL_m(D)$ such that $T^{-1} M T$ is of the form
\begin{equation} \label{strongnileq}
\left(\begin{array}{ccccc}
0_{s_1} & & & & \emptyset \\
A_1 & 0_{s_2} & & & \\ 
& A_2 & \ddots & & \\ 
& & \ddots & 0_{s_{r-1}} & \\
* & & & A_{r-1} & 0_{s_r} 
\end{array} \right)
\end{equation}
where $0_{s_i}$ is the square zero matrix of size $s_i \ge 1$ and $A_i$ has independent columns 
over $D$ for each $i$. In particular, the strong nilpotency index of $M$ does not exceed $m$ in case
$M$ is strongly nilpotent.
\end{theorem}

\begin{proof}
Write $y^{(i)} = (y^{(i)}_1, y^{(i)}_2, \ldots, y^{(i)}_n)$ for each $i$.
\begin{itemize}

\item[\imp]
Assume that $M$ has nilpotency index $r$. Then 
\begin{align}
M|_{x=y^{(r)}} \cdot 
M|_{x=y^{(r-1)}} \cdot \cdots \cdot M|_{x=y^{(2)}} \cdot M|_{x=y^{(1)}} &= 0 \label{eq0} \\
\intertext{but}
M|_{x=y^{(r-1)}} \cdot 
\cdots \cdot M|_{x=y^{(2)}} \cdot M|_{x=y^{(1)}} &\ne 0 \label{neq0}
\end{align}
Thus for some term in $y^{(r-1)}, \ldots,  y^{(2)}, y^{(1)}$, the coefficient matrix 
$C \in \Mat_n(D)$ on the left hand side of (\ref{neq0}) does not vanish, and we obtain from 
(\ref{eq0}) that $M|_{x=y^{(r)}} \cdot C = 0$. Hence also $M \cdot C = 0$ and the columns 
of $M$ are dependent over $D$. Thus if we choose $T \in \GL_n(D)$ and $s_{r} \in \N$ maximal, 
such that the last $s_r$ columns of $M \cdot T$ are zero, then $s_r \ge 1$ and the first
$m - s_r$ columns of $M \cdot T$ are independent over $D$.

Now replace $M$ by $T^{-1} M T$. Then the first $m - s_r$ columns of $M$ become independent 
over $D$ and the last $s_r$ columns of $M$ become zero. 
Furthermore, by the maximality of $s_r$, we have that the first $m - s_r$ rows of the 
left hand side of (\ref{neq0}) become zero, since otherwise we can find another linear dependence 
between the columns of both $M|_{x=y^{(r)}}$ and $M$, namely by taking for $C$ the coefficient matrix of the 
left hand side of (\ref{neq0}) with respect to a term in $y^{(r-1)}, y^{(r-2)}, \ldots, y^{(2)}, y^{(1)}$, 
such that the first $m-s_{r}$ rows of $C$ are not zero everywhere. Hence the left hand
side of (\ref{neq0}) is of the form
\begin{equation} \label{indeq1} 
\left(\begin{array}{cc} 0_{m-s_r} & \emptyset \\
\tilde{A} & 0_{s_r} \end{array} \right)
\end{equation}
where $0_{m-s_r}$ is the square zero matrix of size $m - s_r$ and $\tilde{A} \ne 0$ because of 
(\ref{neq0}).

Let $\tilde{M}$ be the leading principal 
minor matrix of size $m - s_{r}$ of $M$. 
Since the last $s_r$ columns of $M$, and hence also of $M|_{x=y^{(i)}}$
for all $i \ge 2$, are zero, we have that the product
\begin{align} 
\tilde{M}|_{x=y^{(r-1)}} \cdot \tilde{M}|_{x=y^{(r-2)}} \cdot \cdots \cdot \tilde{M}|_{x=y^{(2)}} 
&\cdot \tilde{M}|_{x=y^{(1)}} \label{indeq2}
\intertext{is equal to the leading principal minor matrix of size $m - s_{r}$ of the left hand side
of (\ref{neq0}), which is the submatrix $0_{m-s_r}$ of (\ref{indeq1}).
The product}
\tilde{M}|_{x=y^{(r-2)}} \cdot \cdots \cdot \tilde{M}|_{x=y^{(2)}} 
&\cdot \tilde{M}|_{x=y^{(1)}} \label{indeq3}
\end{align}
is nonzero, because the last $s_r$ columns of $M|_{x=y^{(r-1)}}$ are zero and hence (\ref{indeq3}) 
is a factor of $\tilde{A}$. 

So by induction on $m$, there exists a $\tilde{T} \in \GL_{m-s_r}(D)$
such that $\tilde{T}^{-1} \tilde{M} \tilde{T}$ is of the form
$$
\left(\begin{array}{ccccc}
0_{s_1} & & & & \emptyset \\
A_1 & 0_{s_2} & & & \\ 
& A_2 & \ddots & & \\ 
& & \ddots & 0_{s_{r-2}} & \\
* & & & A_{r-2} & 0_{s_{r-1}} 
\end{array} \right)
$$
where $A_i$ has independent columns over $D$ for each $i$. Since the first $m - s_r$ columns of $M$ 
are independent over $D$ and the last $s_r$ columns of $M$ are zero, the desired result follows.

\item[\pmi] 
Assume that a $T \in \GL_m(D)$ as in theorem \ref{strongnil} exists. 
Then one can prove by induction on $j$ that for all $j \le r$, 
the first $s_1 + s_2 + \cdots + s_j$
rows of $(T^{-1}MT)|_{x=y^{(j)}} \cdot 
(T^{-1}MT)|_{x=y^{(j-1)}} \cdot 
\cdots \cdot (T^{-1}MT)|_{x=y^{(1)}}$ are zero. Hence the strong nilpotency index of $M$ does not 
exceed $r$. Furthermore, the first $s_1 + s_2 + \cdots + s_j$ rows of 
$(T^{-1}MT)|_{x=y^{(j-1)}} \cdot \cdots \cdot (T^{-1}MT)|_{x=y^{(1)}}$ are of the form
$$
\left( \begin{array}{cc} & \emptyset \\
A_{j-1}|_{x=y^{(j-1)}} \cdot \cdots \cdot A_1|_{x=y^{(1)}} & \end{array} \right)
$$
Since the columns of $A_i$ are independent over $D$ for each $i$, we obtain by looking at
coefficient matrices in a similar manner as in $\Rightarrow$ and by induction on $j$, that 
$A_{j-1}|_{x=y^{(j-1)}} \cdot \cdots \cdot A_{1}|_{x=y^{(1)}} \ne 0$ for all $j \le r$.
Consequently, $r$ is the strong nilpotency index of $M$. \qedhere

\end{itemize}
\end{proof}

\subsection*{Application to Jacobian matrices}

In addition to the result we get by taking for $M$ a Jacobian matrix in theorem
\ref{strongnil}, we have the following.

\begin{corollary} \label{strongnilcor}
Assume $H \in K[x]^n = K[x_1,x_2,\ldots,x_n]^n$, where $K$ is any field. 
Write $y^{(i)} = (y^{(i)}_1, y^{(i)}_2, \ldots, y^{(i)}_n)$ for all $i$.
 
Then for each nonzero $r \in \N$, the following statements are equivalent.
\begin{enumerate}

\item[(1)] 
$(\jac H)|_{x=y^{(r)}} \cdot (\jac H)|_{x=y^{(r-1)}} \cdot 
\cdots \cdot (\jac H)|_{x=y^{(1)}} = 0$, i.e.\@ the strong nilpotency index of $\jac H$ 
(in the sense of definition \ref{EHdef}) does not exceed $r$.

\item[(2)] 
There exists a $T \in \GL_n(K)$ 
such that the Jacobian of $T^{-1} H(Tx)$ is of the form
\begin{equation} \label{blocktriang}
\left(\begin{array}{cccc}
0_{s_1} & & & \emptyset \\
& 0_{s_2} & & \\ 
& & \ddots & \\
* & & & 0_{s_{r'}} 
\end{array} \right)
\end{equation}
where $0_{s_i}$ is the square zero matrix of size $s_i \ge 1$ for each $i$
and $r' = \min\{r,n\}$.

\item[(3)]
For all $j$ with $1 \le j \le r$, we have that the Jacobian
with respect to $y^{(1)}$ of
\begin{gather}
(\jac H)|_{x=y^{(r)}} \cdot 
   (\jac H)|_{x=y^{(r-1)}} \cdot \cdots 
   \cdot (\jac H)|_{x=y^{(j+1)}} \cdot \nonumber \\
t^{(j)}H\bigg(y^{(j)}+t^{(j-1)}H\Big(
   \cdots \big(y^{(2)}+t^{(1)}H(y^{(1)})\big) 
   \cdots \Big)\bigg) \label{Hyj}
\end{gather}
vanishes. (Thus (\ref{Hyj}) is constant with respect to $y^{(1)}$ when $\chr K = 0$.)

\item[(4)]
There exists a $j$ with $1 \le j \le r$, such that the Jacobian
with respect to $y^{(1)}$ of
\begin{gather}
(\jac H)|_{x=y^{(r)}} \cdot 
   (\jac H)|_{x=y^{(r-1)}} \cdot \cdots 
   \cdot (\jac H)|_{x=y^{(j+1)}} \cdot \nonumber \\
H\bigg(y^{(j)}+H\Big(
   \cdots \big(y^{(2)}+H(y^{(1)})\big) 
   \cdots \Big)\bigg) \label{Hyjw}
\end{gather}
vanishes. (Thus (\ref{Hyjw}) is constant with respect to $y^{(1)}$ when $\chr K = 0$.)

\end{enumerate}
\end{corollary}

\begin{proof}
Since (3) \imp (4) is trivial, the following remains to be proved.
\begin{description}

\item[(1) \imp (2)] Assume (1). From theorem \ref{strongnil}, it follows that there exists a 
$T \in \GL_n(K)$ such that $T^{-1} \cdot \jac H \cdot T$ is of the form (\ref{blocktriang}) 
for some $r' \le r$ and hence also for $r' = \min\{r,n\}$. By the chain rule, we obtain
\begin{equation} \label{chainrule}
\jac \big(T^{-1} H(Tx)\big) = T^{-1} \cdot (\jac H)|_{x=Tx} \cdot T = 
(T^{-1} \cdot \jac H \cdot T)|_{x=Tx}
\end{equation}
which gives (2).

\item[(2) \imp (1)] Assume (2). From (\ref{chainrule}), it follows that $T^{-1} \cdot \jac H \cdot T$ 
is of the form (\ref{blocktriang}) as well. Subsequently, (1) can be proved in a similar manner as \pmi of 
theorem \ref{strongnil}.

\item[(1) \imp (3)] Assume (1). By taking the Jacobian with
respect to $y^{(1)}$ of (\ref{Hyj}) 
and dividing by $t^{(j)}t^{(j-1)} \cdots t^{(1)}$, we obtain 
$(\jac H)|_{x=y^{(r)}+\cdots} \cdot 
(\jac H)|_{x=y^{(r-1)}+\cdots} \cdot 
\cdots \cdot (\jac H)|_{x=y^{(1)}+\cdots}$. This product is zero because
it is a homomorphic image of the left hand side of the equality in (1), and (3) follows.

\item[(4) \imp (1)] Assume (4). By taking the Jacobian with
respect to $y^{(1)}$ of (\ref{Hyjw}), we obtain 
$(\jac H)|_{x=y^{(r)}+\cdots} \cdot (\jac H)|_{x=y^{(r-1)}+\cdots} 
\cdot \cdots \cdot (\jac H)|_{x=y^{(1)}+\cdots}$, which is zero on assumption.
Furthermore, the left hand side of (1) is a homomorphic image of it, and (1) follows.
\qedhere

\end{description}
\end{proof}

\noindent
Since all differentiated factors are commuted to the right in the chain rule, this rule 
does not apply in a noncommutative context. Therefore, (3) and (4) of corollary 
\ref{strongnilcor} cannot be generalized to division rings. We will use corollary \ref{strongnilcor}
in the last section.

\subsection*{Comparing regular and strong nilpotency index}

Before constructing some maps $H$ with strongly nilpotent Jacobians,
such that the nilpotency index of the Jacobian is less than the strong nilpotency index,
we formulate a proposition. 

\begin{proposition} \label{idxs}
Let $D$ be a division ring and suppose that $M \in \Mat_m(D[x]) = \Mat_m(D[x_1,x_2,\ldots,x_n])$ is 
strongly nilpotent with index $r$ (in the sense of definition \ref{EHdef}),
such that $r$ exceeds the regular nilpotency index. Then $3 \le r \le m-1$.
\end{proposition}

\begin{proof}
If the regular nilpotency index equal one, then $M = 0$ and we have $r = 1$ as well,
which contradicts the assumptions.
So $r \ge 3$. Following the proof of $\Leftarrow$ of theorem \ref{strongnil}, we
see that $M^{r-1} = 0$ gives $A_{r-1} \cdot A_{r-2} \cdot \cdots A_1 = 0$ for nonzero matrices 
$A_i$, which consist of only one entry when $r = m$. This is a contradiction, so $r \le m-1$.
\end{proof}

\noindent
Using (2) of corollary \ref{strongnilcor}, one can construct a map $H$ in dimension $n = 4$,
such that $\jac H$ has strong nilpotency index $r = 3$ (in the sense of definition \ref{EHdef}) 
and regular nilpotency index two. Take for instance the (cubic) map
$$
H = (0, x_1^2, x_1^3, 3x_2 x_1^2 - 2x_3 x_1)
$$
in dimension four, or the (cubic) homogeneous map
$$
H = (0, 0, x_1^3, x_1^2 x_2, x_1 x_2^2, x_3 x_2^2 - 2x_4 x_1 x_2 + x_5 x_1^2)
$$
in dimension six. By proposition \ref{idxs}, we see that there is no other combination of 
strong nilpotency index $r$ and matrix dimension $m$, such that $n \le 4$ and $r$ exceeds
the regular nilpotency index. 

The dimension $n = 6$ of the second 
homogeneous map $H$ is minimal as well, at least if the base ring $K$ is a field,
but only for $r = 3$.
The minimality of $n$ for $r = 3$ can be seen as follows. If a homogeneous map
$H$ of degree $d$ in dimension $n \le 5$ with similar properties would exists, 
then by $\jac H \cdot \jac H = 0$, we have 
$\rk \jac H \le \lfloor n/2 \rfloor \le \lfloor 5/2 \rfloor = 2$. Hence $\rk A_2 = \rk A_1 = 1$. 
By homogeneity of $H$, the relations between the rows of $A_1$ decompose into linear relations, 
see e.g.\@ the proof of \cite[Lm.\@ 3]{chengrk2}. Consequently, the image of $A_1$ is determined 
by linear contraints and hence generated by vectors $v$ over $K$.  

Since $\rk A_1 = 1$, the image of $A_1$ is generated by exactly one vector $v$ over $K$.
For this vector $v$, we have $C \cdot v = 0$
for all coefficient matrices $C$ of $A_2$ in case $A_2 \cdot v = 0$. Following the proof of 
$\Leftarrow$ of theorem \ref{strongnil}, we see that $(\jac H)^2 = 0$ gives $A_2 \cdot A_1 = 0$, 
so $A_2 \cdot v = 0$. Hence $C \cdot v = 0$ and $C \cdot A_1 = 0$ for all coefficient matrices 
$C$ of $A_2$. Following the proof of $\Leftarrow$ of theorem \ref{strongnil} again, we get
a contradiction with $r = 3$. 

There is however a homogeneous
map in dimension $n = 5$, from which the Jacobian $M$ is strongly nilpotent with index $r = 4$
and regular nilpotency index three, namely
$$
H = (0, x_1^d, x_1^{d-1} x_2, x_1^{d-2} x_2^2, 2x_1^{d-2} x_2 x3-x_1^{d-1} x_4)
$$
By proposition \ref{idxs} and the study of the case $r = 2$ above, we see that there is no other 
combination of strong nilpotency index $r$ and dimension $n$, such that $M$ is
a homogeneous Jacobian of a polynomial map over a field, $n \le 5$ and $r$ exceeds
the regular nilpotency index. 

\subsection*{Noncommutative polynomials}

Notice that in theorem \ref{strongnil}, one may also assume that the indeterminates
of $x$ are noncommutative, provided the indeterminates of $y^{(j)}$ are
noncommutative for each $j$ as well.

In \cite{Yun}, the author Jie-Tai Yu proves that a homogeneous nilpotent matrix $M$ of 
polynomials over a field $K$ is linearly triangularizable if there exist
a (homogeneous) nilpotent matrix of noncommutative polynomials (free algebra)
which corresponds to $M$. This result can be generalized as follows.

\begin{theorem}
Let $M$ be a matrix of size $m \times m$ of polynomials in noncommutative indeterminates 
$x = (x_i \mid i \in I)$ of (weighted) degree $d$ over a division ring $D$. 

Then $M$ has nilpotency index $r$, if and only if $M$ has strong nilpotency index $r$
(in the sense of definition \ref{EHdef}),
if and only if there exist a $T \in \GL_m(D)$ such that
$T^{-1} M T$ is of the form of (\ref{strongnileq}) in theorem \ref{strongnil}.

In the definition of strong nilpotency (in definition \ref{EHdef}), for each $i,j,k,l$,
it does not matter whether $y^{(j)}_i y^{(l)}_k$ and $y^{(l)}_k y^{(j)}_i$ are equalized
(by commutativity assumptions) or not in case $j \ne l$. It is only necessary that
$y^{(j)}_i y^{(j)}_k$ and $y^{(j)}_k y^{(j)}_i$ are not equalized (by commutativity assumptions). 
So $y^{(j)}_i y^{(j)}_k$ and $y^{(j)}_k y^{(j)}_i$ are equal, if and only if $i = k$.
\end{theorem}

\begin{proof}
For simplicity in writing, we assume that all entries of $M$ are homogeneous of degree $d$
instead of weighted homogeneous of degree $d$.

Suppose that $y^{(j)} = (y^{(j)}_i \mid i \in I)$ are partially noncommutative polynomial 
indeterminates satisfying the above noncommutativity properties.
The noncommutative terms of degree $rd$ in $x$ correspond to (partially)
noncommutative products of $r$ terms 
of degree $d$ in $y^{(r)}, \allowbreak y^{(r-1)}, \ldots, y^{(1)}$, in that order, by
\begin{align*}
x_{i_1}x_{i_2} \cdots x_{i_{rd}} \; \mapsto ~
& y^{(r)}_{i_1}y^{(r)}_{i_2} \cdots y^{(r)}_{i_d} \cdot \\
& y^{(r-1)}_{i_{d+1}}y^{(r-1)}_{i_{d+2}} \cdots y^{(r-1)}_{i_{2d}} \cdot \\
& \qquad \vdots \\
& y^{(1)}_{i_{(r-1)d+1}}y^{(1)}_{i_{(r-1)d+2}} \cdots y^{(1)}_{i_{rd}} 
\end{align*}
and each entry of $M^r = 0$ corresponds to the same entry of $M|_{x=y^{(r)}} \cdot 
M|_{x=y^{(r-1)}} \cdot \cdots \cdot M|_{x=y^{(1)}}$ in that manner, whence the latter matrix 
product vanishes as well as the former.
 
The rest of the proof is similar to that of theorem \ref{strongnil}.
\end{proof}

\noindent
Taking the Jacobian of $H = (0,x_1 x_2 - x_3, x_1^2 x_2 - x_1 x_3)$, we get
$$ 
M = \left(\begin{array}{ccc}
0 & 0 & 0 \\
x_2 & x_1 & -1 \\
2 x_1 x_2 & x_1^2 & - x_1
\end{array} \right)
$$
even if we see the components of $H$ as noncommutative polynomials. 
Since the trace of $M \cdot M_{x=y}$ equals 
$$
0 + (x_1 y_1 - y_2) - (x_2 + x_1 y_1) = -y_2 - x_2 \ne 0
$$
we see that $M$ is not strongly nilpotent.
But $M^3 = 0$ holds, regardsless of whether we see the entries as commutative or noncommutative 
polynomials over $\C$. Hence the above theorem does not hold when $M$ is not homogeneous.

\section{Strongly nilpotent matrices in the sense of \cite{mo}}

In \cite{Yu}, the notion of strong nilpotency is generalized as follows. Instead of 
looking at Jacobians with polynomial entries, the author Jie-Tai Yu looks at arbitrary
matrices of functions of a set $S$ to a division ring $D$. Not all polynomials can be 
described as the functions they represent, e.g.\@ taking the first variable over a finite field
does not describe $x_1$ because the polynomial $x_1^q$ corresponds to the identity
as well over a finite field with $q$ elements. But over an infinite field, 
different polynomials correspond to different functions. 

For such matrices of functions, Yu proves that they are linearly 
triangularizable over $D$ when $M$ is so called {\em generalized strongly 
nilpotent}, which means that for some $r \in \N$,
\begin{equation} \label{yugen}
M|_{v_1} \cdot M|_{v_2} \cdot \cdots \cdot M|_{v_r} = 0
\end{equation}
for all $v_i \in S$, where $f|_g$ means substituting $g$ in the entries of $f$.
Again, we call the minimum $r$ such that (\ref{yugen}) holds the {\em strong nilpotency index}
of $M$. We can generalize Yu's result in a similar manner as in theorem \ref{strongnil}.

\begin{theorem} \label{strongnilyu}
Let $m \ge 1$ and $M \in \Mat_m(D^S)$ be a matrix of functions from a set $S$ to a division ring $D$. 

Then $M$ has strong nilpotency index $r$ (in the sense of definition \ref{MOdef}, 
but with functions that do not need to be polynomials), if and only if 
there exists a $T \in \GL_m(D)$ such that $T^{-1} M T$ is of the form
$$
\left(\begin{array}{ccccc}
0_{s_1} & & & & \emptyset \\
A_1 & 0_{s_2} & & & \\ 
& A_2 & \ddots & & \\
& & \ddots & 0_{s_{r-1}} & \\
* & & & A_{r-1} & 0_{s_r} 
\end{array} \right)
$$
where $0_{s_i}$ is the square zero matrix of size $s_i \ge 1$ and $A_i$ has independent columns 
over $D$ for each $i$. In particular, the strong nilpotency index of $M$ does not exceed $m$ 
in case $M$ is strongly nilpotent.
\end{theorem}

\begin{proof} 
Since the proof is essentially similar to that of theorem \ref{strongnil}, we only describe 
a part of the implication $\Rightarrow$.

Assume that $M$ has strong nilpotency index $r$. Then
\begin{align}
M|_{v_r} \cdot M|_{v_{r-1}} \cdot \cdots \cdot M|_{v_2} \cdot M|_{v_1} = 0 \label{meq0}
\intertext{for all $v_1,v_2,\ldots,v_{r-1},v_r \in S$, but}
M|_{v_{r-1}} \cdot \cdots \cdot M|_{v_2} \cdot M|_{v_1} \ne 0 \label{mneq0}
\end{align}
for certain $v_1,v_2,\ldots,v_{r-1} \in S$. Thus for these $v_1,v_2,\ldots,v_{r-1} \in S$,
the left hand side of (\ref{mneq0}) becomes a nonzero matrix $C \in \Mat_m(D)$, 
and we obtain from (\ref{meq0}) that $M|_{v_r} \cdot C = 0$ for all $v_r \in S$. 
\end{proof}

\noindent
Notice that the strong nilpotency index does not exceed the dimension $m$ of the matrix in 
both theorem \ref{strongnil} and theorem \ref{strongnilyu}. We will prove below that this 
property still holds when the base division ring $D$ is replaced by a reduced 
commutative ring $R$, see also \cite[Exercise 7.4.2]{Arnobook}. 
If $R$ is a ring with unity which is not reduced, then there exists an $\epsilon \in R$ such that 
$\epsilon^2 = 0$, and for the matrix $M$ defined by
$$
M := \left( \begin{array}{ccccc}
\epsilon & 0 & \cdots & 0 & 0 \\
1^2 & & & \emptyset & 0 \\
& 1^3 & & & \vdots \\
& & \ddots & & 0 \\
\emptyset & & & 1^m & 0
\end{array} \right)
$$
(where the powers of $1$ are only taken to indicate that the size is $m$)
we have that $M^k$ is obtained from $M$ by shifting all rows $k-1$ places to below,
where zero rows are shifted in above for each row that is shifted out below.
Hence $M^m$ is zero except for the lower left corner entry which equals $\epsilon$,
and the both the regular and the strong nilpotency index of $M$ equal $m + 1$.

\begin{proposition} \label{strongnilyured}
Let $M$ be a square matrix of size $m$ of functions from a set $S$ to a 
reduced commutative ring $R$ with unity. If there exists an $r \in \N$ such that
$M|_{v_r} M|_{v_{r-1}} \cdots M|_{v_1} = 0$ for all $v_i \in S$, 
then $M|_{v_m} M|_{v_{m-1}} \cdots M|_{v_1} = 0$ for all $v_i \in S$.
\end{proposition}

\begin{proof}
Suppose that there exists an $r \in \N$ such that
$M|_{v_r} M|_{v_{r-1}} \cdots M|_{v_1} = 0$ for all $v_i \in S$.
Since the intersection of all prime ideals of $R$ is zero, it suffices to show 
that for every prime ideal $\p$ of $R$,
all entries of $M|_{v_m} M|_{v_{m-1}} \cdots M|_{v_1}$ are contained in $\p$
for all $v_i \in S$. For that purpose, we  
look at residue classes of $\p$ and hence replace $R$ by $R/\p$, which is an
integral domain. Let $K$ be the field of fractions of $R/\p$. By theorem 
\ref{strongnilyu}, $M$ is linearly triangularizable over $K$, which gives
the desired result.
\end{proof}

\noindent
For reduced noncommutative rings, proposition \ref{strongnilyured} can also
be reduced to the domain case, but domains can not always be embedded in
division rings. Therefore, we cannot prove the following.

\begin{conjecture}
Let $M$ be a square matrix of size $m$ of functions from a set $S$ to a reduced 
noncommutative ring $R$. If there exists an $r \in \N$ such that
$M|_{v_r} M|_{v_{r-1}} \cdots M|_{v_1} = 0$ 
for all $v_i \in S$, then $M|_{v_m} M|_{v_{m-1}} \cdots M|_{v_1} = 0$ for
all $v_i \in S$.
\end{conjecture}

\section{Quasi-translations with strong nilpotency index two}

A quasi-translation is a polynomial map $x + H$ whose inverse is $x - H$.
In 1876 in \cite{gornoet}, the authors proved that if $x + H$ is a quasi-translation
such that $H$ is homogeneous and $\rk \jac H \le 2$, then $x + g^{-1} H$ is
a quasi-translation such that $\jac (g^{-1}H)$ satisfies the second and hence any of the 
properties of corollary \ref{strongnilcor} with strong nilpotency index two, 
where $g = \gcd\{H_1, H_2, \ldots, H_n\}$. Actually, the authors look at polynomial
maps $x + H$ such that $\jac H \cdot H = 0$, but that is the same as that $x + H$
is a quasi-translation on account of \cite[Prop.\@ 1.1]{qt}. In \cite{zwang}, the author 
proved a similar result as above for non-homogeneous quasi-translations in dimension three.

At the end of this section, we will show for two other types of quasi-translations $x + H$ 
that $\jac H$ itself satisfies the properties of corollary \ref{strongnilcor} with strong nilpotency 
index two. Quasi-translations are important in the study of polynomials whose Hessian determinant vanishes, 
see \cite{gornoet} or the beginning of section 2 of \cite{singhess}.

Theorem \ref{qt2} below shows that theorem \ref{strongnil} and corollary \ref{strongnilcor}
can be used to find equivalent properties for quasi-translations to have 
strong nilpotency index two. It was the study of such quasi-translations which caused the
author to write this article.

\begin{theorem} \label{qt2}
Assume $x + H$ is a quasi-translation over a field $K$ of characteristic zero.
Then the following statements are all equivalent to $\jac H \cdot \jac H|_{x=y} = 0$.
\begin{enumerate}

\item[(1)] $\jac H \cdot H(y) = 0$.

\item[(2)] There exists a $T \in \GL_n(K)$ and an $s$ with $0 \le s < n$ 
such that for $\hat{H} := T^{-1} H(Tx)$, we have $\hat{H}_i = 0$ for all 
$i \le s$ and $\hat{H}_i \in K[x_1,x_2,\ldots,x_s]$ for all $i > s$.

\item[(3)] $H(x + t H(y)) = H$.

\item[(4)] $H(x+H(y)) = H$.

\end{enumerate}
\end{theorem}
 
\begin{proof}
By taking the Jacobian of (1) with respect to $y$,  
$\jac H \cdot \jac H|_{x=y} = 0$ follows. Conversely, if $\jac H \cdot \jac H|_{x=y} = 0$,
then by integration with respect to $y$, we have $\jac H \cdot H(y) = \jac H \cdot H(0)$.
By substituting $y$ by $x$ on both sides, we obtain $\jac H \cdot H = 
\jac H \cdot H(0)$, and both equalities combine to $\jac H \cdot H(y) = \jac H \cdot H$.
Hence the equivalence of (1) and $\jac H \cdot \jac H|_{x=y} = 0$ follows from $\jac H \cdot H = 0$, 
which is (2) in the proof of \cite[Prop.\@ 1.1]{qt}.
By subsituting $x = y^{(2)}$ and $y = y^{(1)}$, and vice versa, we see that 
$\jac H \cdot \jac H|_{x=y} = 0$ is equivalent to (1) of corollary \ref{strongnilcor} with $r = 2$.
So (1) is equivalent to (1) of corollary \ref{strongnilcor} with $r = 2$.

We shall show that (2) in turn is equivalent to (2) of corollary \ref{strongnilcor} 
with $r = 2$. It is clear that (2) of corollary \ref{strongnilcor} with $r = 2$ follows from (2).
Conversely, assume (2) of corollary \ref{strongnilcor} with $r=2$. If $\jac H = 0$, then we have (2) 
with $s = 0$, so assume $\jac H \ne 0$. By (2) $\Rightarrow$ (1) of corollary \ref{strongnilcor} and 
$\jac H \ne 0$, we obtain that the strong nilpotency index of $\jac H$ equals two. Hence we may assume 
that $\hat{H} := T^{-1} H(Tx)$ is of the form of (\ref{strongnileq}) in theorem \ref{strongnil}.
By applying the fact that (2) of corollary \ref{strongnilcor} with $r = 2$ implies (1) (of this
theorem), but on $\hat{H}$ instead of $H$, we see that $\jac \hat{H} \cdot \hat{H}(y) = 0$.
Now the independence of the columns of $A_1$ subsequently gives $\hat{H}_1 = \hat{H}_2 = \cdots = 
\hat{H}_{s_1} = 0$, so we have (2) with $s = s_1$ by definition of $\hat{H}$.

In order to show that (3) and (4) are equivalent to (1) and (2) as well, it suffices to show
that (3) of corollary \ref{strongnilcor} with $r=2$ implies (3), that (3) implies (4), and that (4)
implies (4) of corollary \ref{strongnilcor} with $r=2$. The latter follows by subsituting 
$x = y^{(2)}$ and $y = y^{(1)}$, because the Jacobian matrix with respect to $y^{(1)}$ of
$H(y^{(2)})$ is zero. That (3) implies (4) follows by subsituting $t = 1$. So it remains
to show that (3) of corollary \ref{strongnilcor} with $r=2$ implies (3). 

For that purpose, 
suppose we have (3) of corollary \ref{strongnilcor} with $r=2$. By taking $j = 2$, we get 
$t^{(2)} H\big(y^{(2)} + t^{(1)} H(y^{(1)}\big) = H\big(y^{(2)} + t^{(1)} H(0)\big)$, 
which gives $H(x + tH(y)) = H(x + tH(0))$ after suitable substitutions. By substituting $y$ by $x$ on 
both sides, we obtain $H(x + tH(x)) = H(x + tH(0))$, and both equalities combine to $H(x + tH(y)) = 
H(x + tH(x))$. Hence (3) follows from $H(x+tH) = H$, which is (1) in the proof of \cite[Prop.\@ 1.1]{qt}. 
\end{proof}

\noindent
(2) $\Rightarrow$ (3) of the following theorem was proved by E. Formanek in \cite[Th.\@ 4]{for} 
for the case that $H$ has no linear terms, but such a condition is unnecessary.

\begin{theorem} \label{rank1}
Assume $F = x + H$ is a Keller map in dimension $n$ over a field $K$ of characteristic
zero, such that $\rk \jac H = 1$. Then $F$ is invertible and for
\begin{enumerate}

\item[(1)] $H$ has no linear terms,

\item[(2)] $\det \jac F = 1$,

\item[(3)] $\jac H \cdot \jac H|_{x=y} = 0$,

\end{enumerate}
we have (1) $\Rightarrow$ (2) $\Rightarrow$ (3). 
\end{theorem}

\begin{proof}
By \cite[Lm.\@ 3]{for}, we have $H_i \in K[p]$ for some $p \in K[x]$, and
by way of Gaussian elimination on the coefficients of the $H_i$ with respect to $p$,
we see that there exists a $T \in \GL_n(K)$ and an $s \ge 0$ such that for 
$\hat{H} = T^{-1} H(Tx)$, we have $\hat{H}_i \in K$ for all $i \le s$ and
$0 < \deg \hat{H}_{s+1} < \deg \hat{H}_{s+2} < \cdots < \deg \hat{H}_n$.
Furthermore, $\det \jac (x + \hat{H}) \in K^{*}$ and $\hat{H}_i \in K[\hat{p}]$ for all $i$, 
where $\hat{p} = p(Tx)$. We shall show below that $\jac (x + \hat{H})$ is a lower triangular 
matrix (but not necessarily with ones on the diagonal). Consequently, $x + \hat{H}$ is a 
composition of elementary invertible polynomial maps and $F$ is a tame invertible map,
because $\det \jac (x + \hat{H}) \in K^{*}$.

Notice that $\det \jac (x + \hat{H}) - 1$
is the sum of the determinants of all principal minor matrices of $\jac \hat{H}$.
Since $\rk \jac \hat{H} = 1$, all minor determinants of size $2 \times 2$ 
or greater vanish. Hence $\tr \jac \hat{H} = \det \jac (x + \hat{H}) - 1 \in K$. 
Take $i$ maximal, such that
$\parder{}{x_i} \hat{p} \ne 0$. If $i \le s$, then we are in the situation of (2)
of corollary \ref{strongnilcor} and we have (3) and hence also (2), because $\jac H$
is nilpotent and $\det \jac F - 1$ is the sum of the
determinants of all principal minor matrices of $\jac H$. 

Thus assume that $i > s$. Then $\parder{}{x_j}\hat{H}_j = 0$ for all $j > i$, 
$\deg \parder{}{x_j}\hat{H}_j < \deg \hat{H}_j \le \deg \hat{H}_i - \deg \hat{p}$ for all $j < i$,
and $\deg \hat{H}_i - \deg \hat{p} \le \deg \parder{}{x_i}\hat{H}_i$ because 
$\parder{}{x_i}\hat{H}_i$ is divisible by a polynomial in $\hat{p}$
of sufficiently large degree. Hence $\deg \parder{}{x_j}\hat{H}_j < \deg \parder{}{x_i}\hat{H}_i$ 
for all $j \ne i$ and therefore $\deg \parder{}{x_i}\hat{H}_i = \deg \tr \jac \hat{H} \le 0$.
Consequently, $\parder{}{x_j}\hat{H}_j = 0$ for all $j \ne i$ and $\hat{H}_i$ has degree 
one in $\hat{p}$, which in turn has degree one in $x_i = x_{s+1}$ with leading coefficient in $K$
as a polynomial over 
$K[x_1,x_2,\ldots,x_s]$. This contradicts (1). By $\hat{H}_j \in K$ for all $j \le s$ and
$\hat{H}_j \in K[x_1,x_2,\ldots,x_s,x_{s+1}]$ for all $j > s$, we get $\det \jac (x + \hat{H}) = 
1 + \parder{}{x_i}\hat{H}_i \ne 1$, which contradicts (2) by way of (\ref{chainrule}).
\end{proof}

\begin{corollary}
Assume $x + H$ is a quasi-translation over a field $K$ of characteristic zero, such that
$1 \in \{\deg H, \rk \jac H\}$. Then $\jac H$ has strong nilpotency index two, i.e.\@
$\jac H \cdot \jac H|_{x=y} = 0$. 
\end{corollary}

\begin{proof}
Suppose first that $\deg H = 1$. Since $\jac H$ is a constant matrix, the strong nilpotency index and 
the nilpotency index correspond. Furthermore, the part of degree one of $\jac H \cdot H$ equals 
$(\jac H)^2 \cdot x$, and $(\jac H)^2 = 0$ follows. 

Suppose next that $\rk \jac H = 1$. By \cite[Prop.\@ 1.1]{qt}, $\jac H$ is nilpotent, so $\det \jac F = 1$. 
Hence the desired result follows from (2) $\Rightarrow$ (3) of theorem \ref{rank1}.
\end{proof}

\paragraph{Acknowledgment.} The author wishes to thank the referee for his careful reading and
for several valuable comments. \\
This is a preprint of an article submitted for consideration in
Linear and Multilinear Algebra {\copyright} 2012 copyright Taylor \& Francis.

\end{document}